\documentclass[english,11pt]{amsart}
\usepackage{amsopn,amsthm,amsfonts,amsmath,amssymb,amscd}
\usepackage{babel}
\usepackage{amssymb,tikz,wrapfig,pgfplots}
\usepackage[T1]{fontenc}

\newtheorem{Theorem}{Theorem}[section]

\newtheorem{Lemma}[Theorem]{Lemma}
\newtheorem{Proposition}[Theorem]{Proposition}
\newtheorem{Corollary}[Theorem]{Corollary}
\newtheorem{Example}[Theorem]{Example}
\newtheorem{Remark}[Theorem]{Remark}

\newenvironment{Proof*}{{\it Proof.}}

\newcommand{\RR}{\mathbb{R}}

\newcommand{\TT}{\mathbb{T}}

\newcommand{\supp}{{\rm Supp}}

\newcommand{\CP}[1]{{\rm CP}_{#1}\left(\TT\right)}
\newcommand{\CPrk}[1]{{\rm {CPrk}}\left( #1 \right)}
\newcommand{\cc}[1]{{\rm {cc}}\left( #1 \right)}

\newcommand{\diam}[1]{{\rm diam}(#1)}

    \definecolor{david}{rgb}{.2,.8,.4}
    \definecolor{polona}{rgb}{.8,.2,.2}
   \definecolor{todo}{rgb}{.2,.2,.8}

\begin{document}

\title{Bounds for the completely positive rank of a symmetric matrix over a tropical semiring}
\author{David Dol\v zan, Polona Oblak}
\date{\today}

\address{D.~Dol\v zan:~Department of Mathematics, Faculty of Mathematics
and Physics, University of Ljubljana, Jadranska 21, SI-1000 Ljubljana, Slovenia; e-mail: 
david.dolzan@fmf.uni-lj.si}
\address{P.~Oblak: Faculty of Computer and Information Science, University of Ljubljana,
Ve\v cna pot 113, SI-1000 Ljubljana, Slovenia; e-mail: polona.oblak@fri.uni-lj.si}

   \thanks{The authors acknowledge the financial support from the Slovenian Research Agency  (research core funding No. P1-0222)}

\bigskip

\begin{abstract} 
In this paper, we find an upper bound for the CP-rank of a matrix over a tropical semiring, according to the vertex clique cover of the graph prescribed by the positions of zero entries in the matrix. We study the graphs  that beget the  matrices with the lowest possible CP-ranks and prove that any such graph must have its diameter equal to 2.
\end{abstract}

\maketitle 
%\parindent=0cm

%-----------------------------------------------------
%-----------------------------------------------------
\section{Introduction}
%-----------------------------------------------------
%-----------------------------------------------------

\maketitle 

In this paper, we study the completely positive rank of a matrix over the tropical semiring $\TT$, which is the semiring $(\RR \cup \{\infty\}, \oplus, \odot)$, with operations defined by $a \oplus b = \min\{a,b\}$ and $a \odot b = a+b$.
 
For a semiring $S$, we say that a symmetric  $n \times n$ matrix $A$ over $S$ is \emph{completely positive}, if there exists an $n \times r$ matrix $B$ over $S$ such that
$$A=BB^T.$$ 
The minimal possible $r$ in such factorization, is the \emph{CP-rank} of $A$ and it is denoted by $\CPrk{A}$. Equivalently, a matrix $A$ has $\CPrk{A}=r$ if and only if $r$ is the smallest number, such that there  exist vectors
$b_1,b_2,\ldots,b_r \in \TT^n$ with
$$A=\sum_{i=1}^r b_i b_i^T.$$
If matrix $A$ is not completely positive, we denote $\CPrk{A}=\infty$. Note that in \cite{MR2927632}, the authors refer to CP-rank as the symmetric Barvinok 
rank of a matrix.

Note that over  semirings,  all definitions of the rank of a  matrix do not coincide  
as in the case of matrices over real numbers  with standard operations (see e.g.~\cite{MR2121497,MR3351964}). 
Thus, the CP-rank  (which is a special case of a factor rank) is just one of 
many possible semiring matrix ranks. 

For a completely positive $n \times n$ matrix $A$ over the field $\RR$, Drew, Johnson, Loewy \cite{MR1310974} conjectured that  $\CPrk{A}\leq \left \lfloor\frac{n^2}{4}\right \rfloor$ if $n\geq 4$. Twenty years later, the
conjectured upper bound was proved wrong and corrected to $\frac{n^2}{2}$ for all $n \geq 7$ 
\cite{MR3247223,MR3302593}. However, it is still not known what is the tight upper bound and  it transpires that the problem of determining the CP-rank of any given matrix is a difficult problem \cite{MR3414584,MR1986666}.

Let $M_n(S)$ denote the semiring of all $n \times n$ matrices over the semiring $S$.
Over the tropical semiring $\TT$, Cartwright and Chan \cite{MR2927632} proved that 
$\max\left\{n, \left\lfloor \frac{n^2}{4} \right \rfloor\right\}$ 
is the tight upper bound for the CP-rank of a completely positive matrix  $A \in M_n(\TT)$. 
Over the Boolean semiring and the max-min semiring, the same inequality was proved by Mohindru \cite{MR3343670} and Shitov \cite{MR3434516}.

In \cite{MR3386390},  Shaked-Monderer introduced $\CPrk{G}$ to be the maximum
CP-rank of all real matrices with the pattern prescribed by the graph $G$. She proved that the $\CPrk{G}$ is equal to the to the edge clique cover number of
$G$, if and only if $G$ is not a tree and does not contain a triangle.

We follow \cite{MR2927632} to define  $\CPrk{G}$ over the tropical semiring
to be the maximum of CP-ranks of all  completely positive 
matrices $A=(a_{ij}) \in M_n(\TT)$ such that, for $i \neq j$, $a_{ij} = 0$ if and only if 
$\{i,j\} \in E(G)$. 
 (Note that throughout the paper, zero is a real number and not the tropical additive identity, which is $\infty$.) Observe that in $G$, edges correspond to all entries equal to a specific element $0$ distinct from the additive identity in $\TT$. This graph is a subgraph of the weighted graph 
corresponding to a semiring matrix (see for example \cite{MR2389137}), which is also  called  the precedence graph.
%This definition of a graph corresponding to a matrix is a tropical version of the definition of the graph corresponding to the real symmetric matrix, where the edges in the graph correspond to nonzero off-diagonal entries of the matrix.

In this paper, we find an upper bound for the CP-rank of a matrix with regards to the vertex clique cover of the graph prescribed by the positions of zero entries in the matrix. This bound  can be much lower than the bound $\max\left\{n, \left\lfloor \frac{n^2}{4} \right \rfloor\right\}$ from \cite[Theorem 4]{MR2927632}, see Theorem \ref{theta} and Remark \ref{manycases}. We then proceed to apply these results to 0/1 matrices, since it was established in \cite{MR2927632} that  CP-rank of 0/1 matrices is equal to the edge clique cover number of the corresponding graph. We examine the connection between the ranks of 0/1 matrices and  arbitrary matrices with the same 
positions of zero entries. In the last section, we then study the graphs  that beget the matrices with the lowest possible CP-ranks. We prove that any such graph must have its diameter equal to 2, and provide examples that in case of diameter 2 the rank does not seem to be well behaved.

\bigskip

\section{Preliminary results}

In this section, we give the basic definitions and some preliminary results.

%A \emph{semiring} is a set $S$ equipped with binary operations $+$ and $\cdot$ such that $(S,+)$ is a commutative monoid with %identity element 0, and $(S,\cdot)$ is a monoid with identity element 1. 
%In addition, operations $+$ and $\cdot$ are connected by distributivity and 0 annihilates $S$. A semiring is 
%\emph{commutative} if $ab=ba$ for all $a,b \in S$,  \DEF{entire}  if $ab=0$ implies that $a=0$ or $b=0$ and \DEF{antinegative} %(also called an \DEF{antiring}) if $a+b=0$ implies that $a=0$ and $b=0$.
%
%Moreover, the tropical semiring is a commutative entire antinegative division semiring. It is isomorphic via the logarithm map
%to the \emph{max algebra}, that is the semiring of nonnegative reals $\RR_+$, where the addition is the same as
%in the tropical semiring and the  multiplication is  the ordinary multiplication on reals.
%

%Over past decades, the tropical semirings and other tropical structures  were widely investigated.
%Let us mention a few pioneering works \cite{AkiGauGut09, DevSanStur05, pin94, Simon94} that connect the theory of matrices %over classical and
%tropical worlds, and a few recent works \cite{hollings, izhakian, johnson} that study the semigroup structure of tropical matrices.

\bigskip

First, we provide the characterization of completely positive matrices over the tropical semiring. The subset of $M_n(\TT)$ of all completely positive matrices will be denoted by 
$\CP{n}$.

The following lemma is obvious and characterizes matrices of CP-rank equal to 1.

\begin{Lemma}\label{rk1mtx}
A symmetric matrix $A=(a_{ij})\in M_{n}(\TT)$ has $\CPrk{A}=1$ if and only if $a_{i j_1}\odot a_{kj_2} = a_{kj_1} \odot a_{ij_2}$
for all $i,j_1,j_2,k =1,2,\ldots,n$.
(This means that the difference between any two rows of $A$ with finite entries is a vector with all of its entries equal.)
\end{Lemma}

The following lemma characterizes completely positive matrices over the tropical semiring.

\begin{Lemma}{\cite[Proposition 2 and Theorem 4]{MR2927632}} \label{CP}
A symmetric matrix $A=(a_{ij}) \in M_{n}(\TT)$ is completely positive if and only if $2a_{ij} \geq a_{ii}+a_{jj}$ for all $i,j=1,2,\ldots,n$.
\end{Lemma}

This lemma implies that if $a_{ii}=\infty$ for $A=(a_{ij}) \in \CP{n}$ and some $i$, then $a_{ij}=\infty$ for all $j$. Also, if 
all the diagonal elements of a completely positive matrix $A$ are equal to 0, then all off-diagonal
entries are nonnegative. This fact makes it convenient to study such matrices, and also gives sense to studying matrices  defined by the positions of the zero entries.
 The next paragraph describes the procedure to transform the completely positive matrix into a matrix with diagonal entries equal to 0, while preserving the CP-rank.

Choose $A=(a_{ij}) \in M_n(\TT)$.
Let $A[i]\in M_{n-1}(\TT)$ be the matrix obtained from $A$ by deleting its $i$-th row and $i$-th column and let $b[i]\in \TT^{n-1}$ be the vector obtained from vector $b \in \TT^n$ by deleting its $i$-th entry.
If matrix $A$ has $k$ diagonal entries equal to $\infty$, let $C(A) \in M_{n-k}(\TT)$ be the matrix obtained from $A$ by 
 \begin{itemize}\label{CAproc}
  \item deleting $i$-th row and $i$-th column if $a_{ii}=\infty$ for every $i=1,2,\ldots,n$, and
  \item subtracting $\frac{1}{2} a_{ii}$ from each entry in the  $i$-th row and $i$-th column of $A$, if  $a_{ii} \ne \infty$ for every $i=1,2,\ldots,n$. (Note that subtracting a real number from $\infty$ yields $\infty$ 
  and that we subtract $\frac{1}{2} a_{ii}$  twice from $a_{ii}$.)
 \end{itemize}

The next lemma assures us that the rank of a matrix does not change with the above transformation.

\begin{Lemma}\label{C(A)}
If $A\in M_n(\TT)$ is completely positive, then 
$$\CPrk{A}=\CPrk{C(A)}.$$
\end{Lemma}

\begin{proof}
 Let $A=(a_{ij})=\bigoplus_{j=1}^r b_j \odot b_j^T \in CP_{n}(\mathbb{T})$ and suppose first that $a_{ii}=\infty$ for some $i$,
 $1\leq i \leq n$. Observe that  $A[i]=\bigoplus_{j=1}^r b_j[i] \odot b_j[i]^T \in CP_{n-1}(\mathbb{T})$, which implies that
 $\CPrk{A}\geq \CPrk{A[i]}$. Similarly, we can observe that  $\CPrk{A[i]}\geq \CPrk{A}$, by inserting a component equal to $\infty$ to all $b_j$ at the $i$-th component, since a completely positive matrix $A$
 with $a_{ii}=\infty$, by Lemma \ref{CP} must have all entries in the  $i$-th row and $i$-th column equal to $\infty$.

Now, suppose $A=\bigoplus_{j=1}^r b_j \odot b_j^T \in CP_{n}(\mathbb{T})$ and $a_{ii}\ne \infty$ for $i=1,2\ldots,n$. Choose $\alpha \in\RR$, $k \in \{1,2,\ldots,n\}$, and let $B\in CP_{n}(\mathbb{T})$ and 
 $c_j\in \TT^{n}$ be defined as
$$B_{ij}=\begin{cases}
  a_{ij}+2\alpha,& \text{ if } i=j=k,\\
  a_{ij}+\alpha,& \text{ if  either } i=k \text{ or } j=k,\\
  a_{ij},& \text{ otherwise,}
 \end{cases} \;  \text{ and }\; 
  (c_{j})_i=\begin{cases}
   (b_{j})_i+\alpha,& \text{ if } i=k,\\
   (b_{j})_i,& \text{ otherwise.}
 \end{cases}$$
Observe that $B=\bigoplus_{j=1}^r c_j \odot c_j^T $ and thus $\CPrk{B}\leq \CPrk{A}$. By replacing $\alpha$ by $-\alpha$, we obtain $\CPrk{A}\leq \CPrk{B}$. 
By consecutively applying the above procedure with $\alpha=-\frac{1}{2} a_{kk}$ for all $k=1,2,\ldots,n$,  we conclude that $\CPrk{A}=\CPrk{C(A)}$.
\end{proof}

The next example shows that in general, the positions of nonzero entries in a matrix do not determine the CP-rank. We shall see later that this inconvenience can be circumnavigated by replacing $A$ with $C(A)$ as described above, which is a transformation that  preserves the CP-rank by Lemma \ref{C(A)}.

\begin{Example}\label{ExCA}
 Let $$A=\left[\begin{matrix}
   0&1&2\\
   1&2&3\\
   2&3&4
 \end{matrix}\right] =\left[\begin{matrix}
   0\\
   1\\
   2
 \end{matrix}\right] \odot \left[\begin{matrix}
   0 &
   1 &
   2
 \end{matrix}\right] \in \CP{3}.$$
 %Note that by Lemma \ref{rk1mtx}, we have 
 By transformation described on page \pageref{CAproc}, we obtain
$$C(A)=\left[\begin{matrix}
   0&0&0\\
   0&0&0\\
   0&0&0
 \end{matrix}\right]=\left[\begin{matrix}
   0\\
   0\\
   0
 \end{matrix}\right] \odot \left[\begin{matrix}
   0 &
   0 &
   0
 \end{matrix}\right]$$
and have $\CPrk{A}=\CPrk{C(A)} = 1$.

Note that by changing the nonzero entries of matrix $A$, we obtain a matrix with different 
CP-rank. For example,  if
$$B=\left[\begin{matrix}
   0&1&1\\
   1&1&1\\
   1&1&1
 \end{matrix}\right],$$
then $$C(B)=\left[\begin{matrix}
   0&\frac{1}{2}&\frac{1}{2}\\
   \frac{1}{2}&0&0\\
   \frac{1}{2}&0&0
 \end{matrix}\right]=
 \left[\begin{matrix}
   0\\
   \frac{1}{2}\\
   \frac{1}{2}
 \end{matrix}\right] \odot \left[\begin{matrix}
   0&
   \frac{1}{2}&
   \frac{1}{2}
 \end{matrix}\right] \oplus \left[\begin{matrix}
   \infty \\
   0 \\
   0
  \end{matrix}\right] \odot 
   \left[\begin{matrix}
   \infty &
   0 &
   0
 \end{matrix}\right].$$ Lemma \ref{rk1mtx} implies that $\CPrk{B}\ne 1$.
 Note that $\CPrk{C(B)}\leq 2$ and by Lemma \ref{C(A)} it follows that $\CPrk{B}=\CPrk{C(B)}=2$.
\end{Example}

\bigskip

%-----------------------------------------------------
%-----------------------------------------------------
\section{Bounding the CP-rank by the graph structure}
%-----------------------------------------------------
%-----------------------------------------------------

In this section, we find bounds for CP-ranks of matrices with the aid of a graph structure that is prescribed to a given matrix. Namely, we define a graph that corresponds to a matrix (depending on whether different elements of the matrix are equal to zero). We find bounds for the CP-rank of all matrices with a given graph structure. Note that using Lemma \ref{C(A)}, we always work under the assumtpion that $A \in \CP{n}$ has a zero diagonal and nonnegative offdiagonal entries.

Given a symmetric matrix $A=(a_{ij})\in M_{n}(\TT)$, we define $G(A)=(V,E)$ to be a simple
graph with $V=\{1,2,\ldots,n\}$, and for $i \neq j$ we have $\{i,j\}\in E$ if and only if $a_{ij}=0.$ 
Recall that $\CPrk{G}$ is the maximum of CP-ranks of all  symmetric 
matrices $A=(a_{ij}) \in M_n(\TT)$ such that, for $i \neq j$, $a_{ij} = 0$ if and only if 
$\{i,j\} \in E(G)$. 

As usual, in a given graph, the path $a= x_0 \sim x_1 \sim \ldots \sim x_{n-1} \sim x_{n}=b$
connecting vertices $a$ and $b$ has \emph{length} $n$ and the  length of the shortest path connecting
vertices $a$ and $b$
 is called the \emph {distance} between $a$ and $b$ and denoted by $d(a,b)$. We let $d(a,b)=\infty$ if there is no path connecting $a$ and $b$, and we let $d(a,a)=0$. The \emph {diameter} of a graph is a maximal distance between any two of its vertices.
 An \emph{empty graph} is a graph consisting of isolated nodes with no edges.
A \emph{complete graph} on $n$ vertices will be denoted by $K_n$ and a \emph{path} with $n$ vertices will be denoted by $P_n$.
 The \emph{edge clique cover number} $\cc{G}$ of a graph $G$ is the minimal cardinality of the collections of complete
 subgraphs such that every edge of $G$ is in one element of the collection.
 
The following two lemmas give us some bounds for the CP-rank of graphs and their subgraphs.

\begin{Lemma}\label{lemma:induced}
If $H$ is an induced subgraph of the graph $G$, then
 $$\CPrk{H}\leq \CPrk{G}.$$
\end{Lemma}

\begin{proof}
Let $H$ be an induced subgraph of $G$ and suppose without loss of generality that $V(H)=\{1,2,\ldots,m\}$ 
and $V(G)=\{1,2,\ldots,n\}$, $m \leq n$.
Choose any $A \in M_{n}(\TT)$ with $G(A)=G$, and let $B$ be its $m \times m$ leading principal submatrix. It is clear that $G(B)=H$. If
$A=\bigoplus_{i=1}^k a_i \odot a_i^T$, then $B=\bigoplus_{i=1}^k b_i \odot b_i^T$, where $b_i$ is a vector
obtained from $a_i$ by deleting its last $n-m$ components. Hence $\CPrk{B}\leq \CPrk{A}$ and
so $\CPrk{H}\leq \CPrk{G}$.
\end{proof}

Recall that the join $G \vee H$ of graphs $G$ and $H$, is the graph union $G \cup H$ together with all the possible 
edges joining the vertices in $G$ to the vertices in $H$. We show in the next lemma that joining a graph with a single 
vertex does not change the CP-rank.

\begin{Lemma}\label{lemma:vee}
For any graph $G$ and $w$ a  vertex not in $G$, we have
 $$\CPrk{G \vee w}=\CPrk{G}.$$
\end{Lemma}

\begin{proof}
Take any $A \in M_{n+1}(\TT)$ with $G(A)=G\vee w$. Hence,  $A$ is a direct sum of matrix $B \in M_{n}(\TT)$, $G(B)=G$, 
with the size one zero matrix. There exist $b_i \in  \TT^n$, $i=1,2,\ldots,k$, $k \leq \CPrk{G}$, such that 
$B=\bigoplus_{i=1}^k b_i \odot b_i^T$.
Define $a_i=\left[ b_i^T \; 0\right]^T \in \TT^{n+1}$ and observe that $A=\bigoplus_{i=1}^k a_i \odot a_i^T$ and
hence $\CPrk{A} \leq k$. This implies that $\CPrk{G \vee w}\leq \CPrk{G}$.
By Lemma \ref{lemma:induced}, it follows that $\CPrk{G \vee w}=\CPrk{G}$.
\end{proof}

Now, we define the \emph{vertex clique cover} $\gamma$ of a graph $G$ as a collection of $r$ complete subgraphs
such that every vertex of $G$ is in some element of the collection. One can always assume that the vertices 
of $G$ are 
labeled so that
%Given a graph $G$ and a
%vertex clique cover $\gamma$ of graph $G$, let us denote
$$
\gamma = (K_{q_1},K_{q_2},\ldots,K_{q_k},\underbrace{K_1,\ldots,K_1}_l)= (K_{q_1},K_{q_2},\ldots,K_{q_k}, l K_1),$$
  where $q_1\geq q_2 \geq \ldots \geq q_k \geq  2$. Define  the  \emph{vertex clique cover number} of $\gamma$ as
  $$
  \theta(\gamma)= k+  \sum_{i=1}^k (i-1)q_i+kl+\left\lfloor\frac{l^2}{4}\right\rfloor.
  $$
 It is worth noting that 
a vertex clique cover number is the same as a chromatic number of the complement of the graph. In Theorem \ref{theta}, we will prove that CP-rank of a matrix $A$ is bounded by 
 $\theta(\gamma)$ for any  vertex clique cover $\gamma$ of $G=G(A)$.
  
\begin{Example}\label{paw}
  Note that a vertex clique cover is not unique. Let $G$ be a paw graph.
    	
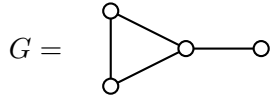
\begin{center}
			\begin{tikzpicture}[style=thick]
				\draw (0,0) -- (-1,-0.5) -- (-1,0.5)  -- (0,0) -- (1,0);
				\draw[fill=white] (0,0) circle (1mm)  (-1,0.5) circle (1mm) (-1,-0.5) circle (1mm) (1,0) circle (1mm);
	   			\draw (-1.5,0) node[anchor=east]{$G=$};	
			\end{tikzpicture}
	\end{center}
 Its vertex clique covers are
 $$ \gamma_1= (K_{3},K_{1}) \; \text{ and } \; \gamma_2 = (2 K_{2}),$$
 so
 $\theta(\gamma_1)=2$ and  $\theta(\gamma_2)=4$.
\end{Example}

The next theorem specifies an upper bound for the CP-rank of a matrix according to the vertex clique cover number of a graph corresponding to the matrix.

\begin{Theorem}\label{theta}
 Choose $A\in \CP{n}$.
 If $G(A)$ is a nonempty graph or $n \geq 5$, then for every vertex clique cover $\gamma$ of $G(A)$, we have
\begin{equation*}
  \CPrk {A}\leq\theta(\gamma) .
\end{equation*}
Otherwise, if $G(A)$ is an empty graph with $n \leq 4$, then
\begin{equation*}
  \CPrk {A}= n.
\end{equation*}
\end{Theorem}

\begin{proof} Suppose first that $G=G(A)$ is nonempty graph. We will construct $n \times n$ matrices 
$A_1$, $A_2$, $A_3$ and $A_4$, $A=A_{1}\oplus A_{2}\oplus A_{3}\oplus A_{4}$, which will correspond to subgraphs of $G$,
and their CP-ranks will be bounded by
$k$, $\sum\limits_{i=1}^k (i-1)q_i$, $kl$ and $\left\lfloor\frac{l^2}{4}\right\rfloor$, respectively.

\begin{enumerate}
\item 
If $k=0$, then $A_1$ is a zero matrix. Suppose $k \geq 1$.
For $i=1,2, \ldots,k$  denote the components of $x^{(i)}\in \mathbb{T}^{n}$ by
$$x_{j}^{(i)}=\left\{\begin{array}{ll}
0, & \text{ if } \ q_{1}+\ldots+q_{i-1}+1\leq j\leq q_{1}+\ldots+q_{i-1}+q_{i},\\
\infty, & \text{ otherwise,}
\end{array}\right.$$
for all $j=1,2, \ldots,n$. 
Define
$$
A_{1}=\bigoplus_{i=1}^{k}x^{(i)}\odot \left(x^{(i)}\right)^{T}
$$
is a sum of $k$ matrices of $\mathrm{C}\mathrm{P}$-rank one.
Note that $A_1$ coincides with $A$ at all elements that correspond to the edges of cliques $K_{q_1}$ to $K_{q_k}$ of $G$ .

\item
%$y^{(i,j,s)}(y^{(i,j,s)})^{T}$ coincides with $A$ in the $s$-th column in the block $(i,\ j)$ . (DEFINE WHAT YOU MEAN BY %BLOCK!!!)
%
If $k \leq 1$, then $A_2$ is a zero matrix. Suppose that  $k \geq 2$.
For $i=1,2, \ldots,k-1$, $j=i+1,i+2, \ldots,k$ and $s=1,2, \ldots,q_{j}$ denote the components of
$y^{(i,j,s)}\in \mathbb{T}^{n}$ by
$$y_{t}^{(i,j,s)}=\left\{\begin{array}{ll}
0, & \mathrm{i}\mathrm{f}\ t=q_{1}+\ldots+q_{j-1}+s,\\
a_{t,q_{1}+\ldots+q_{j-1}+s}, & \mathrm{i}\mathrm{f}\ q_{1}+\ldots+q_{i-1}+1\leq t\leq q_{1}+\ldots+q_{i},\\
\infty, & \text{otherwise}.
\end{array}\right.$$

Let us define
$$
A_{2}=\bigoplus_{i=1j}^{k-1}\bigoplus_{=i+1}^{k}\bigoplus_{s=1}^{q_{j}}y^{(i,j,s)}\odot \left(y^{(i,j,s)}\right)^{T}.
$$
Note that $A_{2}$ is a sum of $\displaystyle \sum_{j=1}^{k}(j-1)q_{j}$ matrices of $\mathrm{C}\mathrm{P}$-rank one, that coincides with the matrix $A$ at all elements that correspond to the edges between any of the cliques $K_{q_1}$ to $K_{q_k}$ of $G$.

\item 
%$z^{(i,j)}(z^{(i,j)})^{T}$ coincides with $A$ in the $q_{1}+\ldots+q_{k}+j$-th column in the $i$-th block
%of rows.
If $k=0$ or $l=0$, then $A_3$ is a zero matrix.
Suppose that $k, l \geq 1$. For $i=1, \ldots, k$ and $j=1, \ldots, l$ denote the components of $z^{(i,j)}\in \mathbb{T}^{n}$ by
$$z_{t}^{(i,j)}=\left\{\begin{array}{ll}
0, & \mathrm{i}\mathrm{f}\ t=q_{1}+\ldots+q_{k}+j,\\
a_{t,q_{1}+\ldots+q_{k}+j}, & \mathrm{i}\mathrm{f}\ q_{1}+\ldots+q_{i-1}+1\leq t\leq q_{1}+\ldots+q_{i},\\
\infty, &\text{otherwise}.
\end{array}\right.$$

Let
$$
A_{3}=\bigoplus_{i=1}^{k}\bigoplus_{j=1}^{l}z^{(i,j)}\odot \left(z^{(i,j)}\right)^{T}
$$
be a matrix defined as a sum of $kl$ matrices of $\mathrm{C}\mathrm{P}$-rank one.
Note that $A_3$ coincides with the matrix $A$ at all elements that correspond to the edges between any of the clique $K_1$ and any of the cliques $K_{q_1}$ to $K_{q_k}$ of $G$.

\item 

If $l \leq 1$, then $A_4$ is a zero matrix and for $l \geq 4$  let the matrix $A_4$ be defined by 
 $$(A_{4})_{ij}=\left\{\begin{array}{ll}
\infty, & \mathrm{i}\mathrm{f}\ i \le q_{1}+\ldots+q_{k} \text{ or } j \le q_{1}+\ldots+q_{k},\\
a_{ij}, & \mathrm{o}\mathrm{t}\mathrm{h}\mathrm{e}\mathrm{r}\mathrm{w}\mathrm{i}\mathrm{s}\mathrm{e}.
\end{array}\right.$$ 
Note that $A_4$ coincides with the matrix $A$ at all elements that correspond to the edges between any of the cliques $K_1$ of $G$.

If $l \geq 4$, then note that $A_4$ can be written as a sum of at most $\left\lfloor \frac{l^2}{4}\right\rfloor$ CP-rank one matrices by \cite[Theorem 4]{MR2927632}.

In the case $2\leq l\leq3$, observe that $n\geq 5$ implies that $k >0$. This further implies that $A_3 \ne 0$, and thus $\left(A_3\right)_{ii} =0$ for $i \geq q_{1}+\ldots+q_{k} +1$, by the construction of $A_3$ above.
For $l=2$, matrix $A_4$ is of CP-rank $\left\lfloor \frac{2^2}{4}\right\rfloor=1$, since $A_4=[\infty,\ldots,\infty,0,a_{n-1,n}]^T \odot [\infty,\ldots,\infty,0,a_{n-1,n}]$.
For $l=3$, assume without loss of generality that $a_{n-1,n}=\max\{a_{n-2,n-1},a_{n-2,n},a_{n-1,n}\}$. In this case, we have 
$A_4= a\odot a^T \oplus b \odot b^T$,
where $a= \left[\begin{matrix}
\infty & \ldots &\infty&0&a_{n-2,n-1}&\infty\end{matrix}
\right]^T\in \TT^n$ and $b=
\left[\begin{matrix}
\infty&\ldots&\infty&a_{n-2,n}&a_{n-1,n}&0\end{matrix}
\right]^T\in \TT^n$. It follows that $\CPrk{A_4}=2=\left\lfloor \frac{3^2}{4}\right\rfloor$.

%=\oplus u_{i}(x_{i})^{T}\oplus\oplus\underline{\mathrm{P}.}\mathrm{y}_{1}\cdot(\mathrm{y}_{1}\cdot)^{\mathrm{T}}
%$$
%$$
%i=1\ i=1
%$$
%Dodaj "se, da so 0 na diagonali "ze narejene po to"cki (3).
\end{enumerate}

Observe that
$$
A=A_{1}\oplus A_{2}\oplus A_{3}\oplus A_{4}
$$
and therefore the inequality in the statement follows. 

If $G$ is an empty graph, then $k=0$. 
In addition, if $n=l\geq 5$, we construct $A_4$ as above, and then $A=A_4$ is a sum of at most $\left\lfloor\frac{n^2}{4}\right\rfloor$ matrices of CP rank one. If $n \leq 4$, then observe that $\left\lfloor\frac{n^2}{4}\right\rfloor \leq n$, so by \cite[Theorem 4]{MR2927632} $A$ can be written as a sum of at most $n$ matrices of CP rank one. However, since $G$ is an empty graph, each summand with CP rank one can have at most one zero element. Since $A=A_4$ has zeroes on the diagonal, this implies that there must be exactly $n$ summands with CP rank one. 
\end{proof}

\begin{Remark}\label{manycases}
 Note that $\theta(\gamma)$ is a much smaller number than $\left\lfloor\frac{n^2}{4}\right\rfloor$ whenever $k \ge 1$, so there are infinite families of graphs and consequently infinite families of matrices for which we have found a much lower bound for their CP rank.
For example, when $k=1$ (and similarly, one can reason for all other $k \geq 1$), $\theta(\gamma)=1+l+\left\lfloor\frac{l^2}{4}\right\rfloor$, which (since $q_1$ can be arbitrarly large) can actually be arbitrarily smaller than $\left\lfloor\frac{(q_1+l)^2}{4}\right\rfloor$.
\end{Remark}

\begin{Example}\label{paw}
 Theorem \ref{theta} implies that any matrix
 $$A=\left[
 \begin{matrix}
  0 & 0 & 0 & a\\
  0 & 0 & 0& b\\
  0 & 0 & 0 & 0\\
  a & b & 0 & 0  
 \end{matrix}
 \right] \in \CP{4},$$
 where $a,b > 0$, which corresponds to the paw graph from Example \ref{paw}, has 
 $\CPrk{A} \leq \theta(\gamma_1) < \theta(\gamma_2)$. Note that by Lemma \ref{rk1mtx} it follows that
 $\CPrk{A} =2$.
\end{Example}

The next example shows that CP-rank of a matrix $A$ with an empty graph can be strictly greater than
 $n$, when $n > 4$.
 
\begin{Example}\label{CPrk6}
 Let $$A=\left[\begin{matrix}
   0&1&1&3&3\\
   1&0&3&1&1\\
   1&3&0&1&1\\
   3&1&1&0&3\\
   3&1&1&3&0 
 \end{matrix}\right] \in \CP{5}$$
 and let us prove that $\CPrk{A} = 6$. 
 
 Suppose there exist vectors $b_1,b_2,\ldots,b_5 \in \TT^5$ such that
  $$A=\bigoplus_{i=1}^5 b_i \odot b_i^T.$$
 Since all diagonal entries of $A=(a_{ij})$ are equal to zero and all offdiagonal entries are nonzero, it follows that each 
 $b_i=[b_{i1},b_{i2},\ldots,b_{i5}]^T$ has nonnegative entries with exactly one zero entry. Without any 
 loss of generality, we asume that $b_{ii}=0$, $i=1,2,\ldots,5$.

 Let us define ${\mathcal E}=\{(1,2),(1,3),(2,4),(2,5),(3,4),(3,5)\}$ the set of indices such that for $k < l$ we have 
  $a_{kl}=1$ if and only if  $(k,l)\in {\mathcal E}$.
   Note that $b_{kl}=b_{kk}+b_{kl} \geq a_{kl}$ for all $1 \leq k < l \leq 5$, which gives us 
   $b_{kl} \geq 1$ for $(k,l) \in {\mathcal E}$ and    $b_{kl} \geq 3$ for $(k,l) \notin {\mathcal E}$.
   
Moreover, for  any pair $(k,l) \notin {\mathcal E}$, $1\leq k<l\leq5$, and any $i$, we have $b_{ki}+b_{il} \geq a_{kl} = 3$. This gives us
  \begin{align}\label{eq:max}
   b_{21}+b_{24} &\geq 3 & b_{31}+b_{34} &\geq 3 & b_{21}+b_{25} &\geq 3 \notag \\
   b_{31}+b_{35} &\geq 3&  b_{12}+b_{13} &\geq 3 & b_{42}+b_{43} &\geq 3\\
   b_{52}+b_{53} &\geq 3 & b_{24}+b_{25} &\geq 3 & b_{34}+b_{35} &\geq 3. \notag
  \end{align}
Note that $b_{ik}+b_{il} \geq 2$ for all $i$ distinct from $k$ and $l$ and thus
   \begin{equation}\label{eq:min}
      \min \{b_{kl},b_{lk}\}=1.
    \end{equation}
for any $(k,l) \in {\mathcal E}$.

  Choose  $(k,l)=(1,2)$ and by $\eqref{eq:min}$ we have $b_{21}=1$ or $b_{12}=1$. In the case
  $b_{21}=1$, we apply \eqref{eq:max} and \eqref{eq:min}
  for several times, to observe that $b_{24} \geq 2$, $b_{25}\geq 2$,  $b_{42}=1$, $b_{43}\geq 2$, $b_{34}=1$, 
  $b_{35}\geq2$, $b_{53}=1$, $b_{52}\geq 2$ and so $b_{25}=1$, a contradiction. Similar arguments give us a 
  contradiction also in the case $b_{12} = 1$.
  Hence, we proved that $\CPrk{A} \geq 6$ and by Theorem \ref{theta}, it follows that  $\CPrk{A} = 6$. 
\end{Example}

%\section{CP-rank of 0/1 matrices}

In the rest of this section, we apply the above results to the study of  the CP-rank of 0/1 matrices over $\TT$. Note again that $0$ and $1$ here  represent real numbers. Equivalently, one could also  study $0/\infty$ matrices, where $0$ and $\infty$ represent  the tropical identity and tropical zero.

It can be seen that CP-rank of a 0/1 matrix $A$ is equal to the edge clique cover number of $G(A)$, denoted by $\cc{G(A)}$ \cite[Proposition 3]{MR2927632}. Note that it was proved that the  edge clique cover number of a graph is equal to
the intersection number of the graph \cite{MR0464695}. Since determining the intersection number is an NP-complete problem \cite{MR0480180}, it seems useful to obtain some easily calculable bounds for the CP-rank of
a 0/1 matrix and the following two propositions offer some results in this direction, by using the same approach as in the proof of Theorem \ref{theta}.

\begin{Proposition}
 If $A\in \CP{n}$ is a 0/1 matrix such that $G(A)$ is an empty graph, then
 \begin{equation*}
  \CPrk {A}=n.
\end{equation*}
 \end{Proposition}

\begin{proof}
 Let us define $v^{(i)} \in \TT^n$, $i=1,2,\ldots,n$, by
  $$v_{j}^{(i)}=\left\{\begin{array}{ll}
      0, &\text{ if } i=j,\\
      1, & \text{ if } i\ne j.
\end{array}\right.$$
It is easy to verify that $$A=\bigoplus_{i=1}^n v^{(i)}\odot \left( v^{(i)} \right)^T$$
and so $\CPrk{A} \leq n$.
Suppose now $A=\bigoplus\limits_{i=1}^{n-1} u^{(i)}\odot \left( u^{(i)} \right)^T$. Since $A$ has $n$ diagonal entries equal to 0, there exists $j$ such that $u^{(j)}_t=u^{(j)}_s=0$ for some $1\leq s,t \leq n$. It follows that $\left( u^{(i)}\odot \left( u^{(i)} \right)^T\right)_{ts}=0$ and thus $a_{ts} \neq 0$, a contradiction. Therefore, 
$\CPrk{A}=n$.
\end{proof}

%\begin{Remark}
Note that the above proposition is not valid for matrices which are not 0/1, as Example \ref{CPrk6} shows.
%\end{Remark}

%
%\begin{Proposition}
%   Let  $A\in \CP{n}$ be a  0/1 matrix and let
%  $\gamma = (K_{q_1},K_{q_2},\ldots,K_{q_k},l K_{1})$, $q_i \geq 2$ for all $i=1,\ldots,k$, be a vertex clique cover 
%  of $G=G(A)$.
%
%  Let $m$ be the number of edges  in $G(A)$ not included in any $K_{q_i}$, $i=1,2,\ldots,k$.  Then
% \begin{equation*}
%  \CPrk {A}\leq k+ l + m.
%\end{equation*}
% \end{Proposition}
%
%\begin{proof}
%For $i=1,2, \ldots,k$ and $j=1,2,\ldots,l$ let us define vectors $v^{(i)}, u^{(j)}\in \mathbb{T}^{n}$ by
%$$v_{t}^{(i)}=\left\{\begin{array}{ll}
%0, & \text{ if } q_{1}+\ldots+q_{i-1}+1\leq t \leq q_{1}+\ldots+q_{i-1}+q_{i},\\
%1, & \text{ otherwise,}
%\end{array}\right.$$
%and
%$$u_{t}^{(j)}=\left\{\begin{array}{ll}
%0, & \text{ if } t= q_{1}+q_2+\ldots+q_{k}+j,\\
%1, & \text{ if } q_{1}+\ldots+q_{k}+1\leq t \leq q_{1}+\ldots+q_{k}+j-1 \\
%     & \text{ or }  t \geq q_{1}+\ldots+q_{k}+j+1,\\
%\infty, & \text{ otherwise,}
%\end{array}\right.$$
%for all $t=1,2, \ldots, n$.  
%Moreover, let ${\mathcal M}\subseteq E(G)$ be the set of edges in $G$, which are not included in any $K_{q_i}$, $i=1,2,\ldots,k$. 
%For every edge $\{p,r\} \in {\mathcal M}$ let $z^{(p,r)}\in \mathbb{T}^{n}$ be defined by
%$$z_{t}^{(p,r)}=\left\{\begin{array}{ll}
%0, & \text{ if } t=p \text{ or } t=q,\\
%1, & \text{ otherwise.}
%\end{array}\right.$$
%Note that
%$$
% A=\bigoplus_{i=1}^{k}v^{(i)} \odot \left(v^{(i)}\right)^{T}\bigoplus_{i=1}^{l}u^{(i)}\odot  \left(u^{(i)}\right)^{T}
% \bigoplus_{\{p,r\} \in {\mathcal M}} z^{(p,r)}\odot \left(z^{(p,r)}\right)^{T}
%$$
%and therefore $  \CPrk {A}\leq k+ l + m$.
%\end{proof}

For any given matrix $A \in M_n(\TT)$, we define its support, $\supp(A) \in M_n(\TT)$, by 
$$\supp(A)_{ij}=\left\{\begin{array}{ll}
      0, &\text{ if } a_{ij}=0,\\
      1, & \text{ if } a_{ij} \ne 0.
\end{array}\right.$$
In Example \ref{ExCA}, we showed that the CP-rank of $A$ and $\supp(A)$ do not necessarily coincide.

\begin{Lemma}\label{cc}
 If $G$ is a graph with  $\CPrk{G}=\cc{G}$, then for every  $A=(a_{ij}) \in \CP{n}$  with $G(A)=G$ choose
 edge clique cover $Q_1,Q_2,\ldots,Q_{\cc{G}}$. Then
  \begin{equation}\label{eq:cc}
          A = \bigoplus_{i=1}^{\cc{G}} b_i \odot b_i^T
        \end{equation}
    and    the following two statements hold:
     \begin{enumerate} 
      \item[(a)]\label{cca}We have a bijective correspondence between the  cliques 
      $Q_1,Q_2,\ldots,Q_{\cc{G}}$ and the 
      summands $b_i$ of the sum, where the vertices of the clique $i$  correspond to the zero entries 
      of $b_i$.
   \item[(b)] If $a_{uv}$ is the minimal nonzero entry in $A$, then for every $i=1,2,\ldots,\cc{G}$ and $j=1,2,\ldots,n$, we have $(b_i)_j =0$ or  $(b_i)_j \geq a_{u,v}$.
%   \item  If $Z_i=\{k_1,\ldots,k_{z_i}\}$ denotes the set of indices of zero entries of $b_i$ and if  $a_{uv} \ne 0$ is the minimal nonzero entry of $A$, then $(b_i)_u=a_{uv}$ for some $i$, where $u$ is a 
% neighbour of a vertex in $Z_i$ or  $(b_i)_v=a_{uv}$ for some $i$, where $v$ is a   neighbour of a vertex in $Z_i$.
\end{enumerate}
\end{Lemma}

\begin{proof}
 Since $G(A)=G$ and $\CPrk{G}=\cc{G}$, we know that $A = \bigoplus_{i=1}^{\cc{G}} b_i \odot b_i^T$ for some vectors $b_i \in \TT^n$.
For every clique $Q$ from the clique cover $Q_1,Q_2,\ldots,Q_{\cc{G}}$, we have 
$a_{jk}=0$ for all $j,k, \in Q$.
% and $k$ that correspond to vertices belonging to the clique $Q$. 
This implies that there exists $i$  such that $(b_i)_j=(b_i)_k=0$ 
for all $j,k \in Q$.
 The fact that the number of summands of rank one matrices is exactly equal to $\cc{G}$, implies that
  for every clique $Q_i$ in $G$, there exists some vector $b_i$ with components equaling zero at least at all positions corresponding to the vertices of  clique $Q_i$.   
By Lemma \ref{C(A)} and the definition of operations in $M_n(\TT)$, we know that all positions that correspond to vertices outside clique $Q_i$, have to be nonzero. This yields the desired bijective correspondence.

 Now, suppose $a_{uv}$ is the minimal nonzero entry in $A$ and choose $i$ such that $(b_i)_j >0$. By the above, $b_i$ corresponds to a clique $Q_i$ in $G$, so there exist indices $k_1,k_2, \ldots, k_r$ such that  $(b_i)_{k_t} =0$ for all $t=1,2,\ldots,r$ and
 $j \notin Q_i$. Then $a_{jk_t} \leq (b_i \odot b_i^T)_{j,k_t}=(b_i)_j$  for all $t=1,2,\ldots,r$. Since vertices corresponding to $j$ and $k_t$ do not belong to the same clique, there exists at least one $t$ such that $a_{jk_t} \neq 0$, and therefore $a_{uv} \leq a_{jk_t} \leq (b_i)_j$.  
\end{proof}

%
%The following Lemma illuminates the link between the CP-ranks of 0/1 matrices and general matrices. 
%This implies that for finding lower bounds for the CP-rank of any matrix it suffices to study the CP-rank of its corresponding support, which is a 0/1 matrix.
%
%
%\begin{Lemma}
% For any matrix $A \in M_n(\TT)$ we have $\CPrk {A}\geq \cc{G(A)}=\CPrk {\supp(A)}$.
%\end{Lemma}
%\begin{proof}
%  Note that all diagonal entries of $A$ are zero by Remark \ref{entries}. Now, if 
% $A = \bigoplus_{i=1}^{\cc{G}} b_i \odot b_i^T$, every edge of every clique with at least two vertices in $G(A)$ has to correspond to some vector $b_i$ with zeroes at the corresponding positions, and the same is true for every clique in $G(A)$ that consists of only one vertex. But no edges from two different cliques can be represented in the same vector $b_i$, unless these two edges are in itself both part of some other clique. This implies that there are at least as many summands as there are cliques in $G(A)$.
%\end{proof}

By \cite[Proposition 3]{MR2927632}, we have that the CP-rank of \supp(A), which is a 0/1 matrix, is equal to the edge clique cover number of $G(A)$. Therefore it follows that in order to find 
lower bounds for the CP-rank of any matrix, it suffices to study the CP-rank of its corresponding support as the following shows.

\begin{Corollary}\label{ineq}
For any matrix $A \in \CP{n}$ we have 
 \begin{equation*}
   \CPrk {G(A)}\geq  \CPrk {A}\geq \CPrk {\supp(A)}=\cc{G(A)}.
 \end{equation*}
 \end{Corollary}

Note that Example \ref{ExCA} shows that the inequality in Corollary \ref{ineq} is not necessarily true if we omit the condition $A=C(A)$.

%
%\begin{Lemma}
% Suppose $3\leq |G|\leq 6$.
% 
% $ccn(G)=n^2/4$ if and only if $G=K_{p,q}$  
%\end{Lemma}
%
%\begin{proof}
% \po{todo.}
%\end{proof}
%
%\begin{Corollary}
%  Choose a 0/1 matrix $A\in \CP{n}$. 
%  Define oznake.
%  If $k \ne 0$ and  $l \geq 3$, then for every vertex clique cover $\gamma$ of $G(A)$, we have
%\begin{equation*}
%  \CPrk {A}<\theta(\gamma) .
%\end{equation*}
%\end{Corollary}

%\begin{Conjecture} 
 %For every vertex clique cover $\gamma$ {\it of a graph} $G$ {\it with} $l\geq 5$, {\it and every} $ A\in$
%
%$cp()$ {\it with} $G(A)=G$, {\it we have}
%
%CPrk $(A)<\theta(\gamma)$ .
%\end{Conjecture}

\bigskip

\section{Graphs with CP-rank equal to the clique cover number}

In Lemma \ref{cc}, we proved that the lower bound for CP-rank of a graph is its clique cover number.
Therefore, we now proceed by studying the graphs that define matrices with the 
CP-ranks that are as close as possible to the bound from Corollary \ref{ineq}.

%Max Triangle-free=min-diam-2 \cite{MR1322083}.

The following theorem shows that if we aspire to characterize graphs with the lowest possible 
CP-ranks, we can limit ourselves to graphs which are very well connected, i.e. their diameters are at most 2. However, the situation in the case $\diam{G} \leq 2$ appears to be quite complex. We provide examples of acyclic and cyclic graphs with diameter 2 where either $\CPrk{G}=\cc{G}$ 
or  $\CPrk{G}>\cc{G}$.

\begin{Theorem}
 If $G$ is a connected graph with  $\CPrk{G}=\cc{G}$, then  $\diam{G} \leq 2$.
\end{Theorem}

\begin{proof}
 Suppose $\CPrk{G}=\cc{G}$ and $\diam{G} \geq 3$. Thus there exist vertices $u,v \in V(G)$ with $d(u,v)\geq 3$. 
% Since $u$ and $v$ have no common neighbours in $G$, for all edges $\{u,s\},\{v,t\} \in E(G)$ it follows that $\{u,t\}$ and $\{v,s\}$ are not in $E(G)$.

 Define $A=(a_{ij})\in M_n(\TT)$ by 
  $$a_{ij}=\begin{cases}
     0, & \text{if } \{i,j\} \in E(G) \text{ or } i=j,\\
     1, & \text{if } \{i,j\}=\{u,v\},\\
     2, & \text{if } \{i,j\}\notin E(G) \text{ and } i\ne j.
   \end{cases}
  $$ 
 Observe that $G(A)=G$ and $A \in \CP{n}$. Let $A$ be of the form  \eqref{eq:cc}. Since $a_{u,v}=1$ is the minimal nonzero entry of $A$ and $(b_i \odot b_i^T)_{u,v} =(b_i)_u+(b_i)_v= 1$ for some $i$, then
 by Lemma \ref{cc} (b), $(b_i)_u=0$ and $(b_i)_v=1$ or $(b_i)_v=0$ and $(b_i)_u=1$. Suppose without loss of generality that $(b_i)_u=0$ and $(b_i)_v=1$. By  Lemma \ref{cc} (a), $(b_i)_l=0$ for some 
 $l \ne u$ and thus  $a_{v,l} \leq (b_i \odot b_i^T)_{v,l} =(b_i)_v+(b_i)_l= 1$. Hence by definition of $A$, 
 $\{v,l\} \in E(G)$, which contradicts $d(u,v) \geq 3$. 
 %So $\CPrk{A}>\cc{G}$ and thus $\CPrk{G}>\cc{G}$, a contradiction.
\end{proof}

\begin{Example}
 If $G=P_3$ is a path on 3 vertices, then all matrices $A \in \CP{3}$ with $G(A)=P_3$ have (up to a permutational conjugation) the form
 $$A=\left[
 \begin{matrix}
  0 & a & 0\\
  a & 0 & 0\\
  0 & 0 & 0  
 \end{matrix}
 \right]=\left[
 \begin{matrix}
  a \\
  0\\
  0  
 \end{matrix}
 \right] \odot \left[
 \begin{matrix}
  a &
  0&
  0  
 \end{matrix}
 \right] \oplus \left[
 \begin{matrix}
  0 \\
  \infty\\
  0  
 \end{matrix}
 \right] \odot \left[
 \begin{matrix}
  0 &
  \infty&
  0  
 \end{matrix}
 \right]$$
 for some $0 \ne a \in \TT$. By Lemma \ref{rk1mtx}, $\CPrk{A} \ne1$, so it follows that $\CPrk{A}=2$ and thus  $\CPrk{P_3}=\cc{P_3}=2$.
\end{Example}

\begin{Example}
 If $G$ is a paw graph (see Example \ref{paw}), then $\cc{G}=2$. Since every matrix $B \in \CP{4}$, $G(B)=G$, has (up to permutational conjugation) the form
 $$B=\left[
 \begin{matrix}
  0 & a & b & 0\\
  a & 0 & 0& 0\\
  b & 0 & 0 & 0\\
  0 & 0 & 0 & 0  
 \end{matrix}
 \right]=\left[
 \begin{matrix}
  \infty \\
  0\\
  0  \\
  0
 \end{matrix}
 \right] \odot \left[
 \begin{matrix}
  \infty &
  0&
  0  & 0
 \end{matrix}
 \right] \oplus \left[
 \begin{matrix}
  0 \\
  a\\
  b\\
  0  
 \end{matrix}
 \right] \odot \left[
 \begin{matrix}
  0 &
  a&
  b & 0  
 \end{matrix}
 \right],$$
 for some $0 \ne a,b \in \TT$.  By Lemma \ref{rk1mtx}, we have $\CPrk{B} \ne1$ and so it follows that  $\CPrk{G}=\cc{G}=2$.
\end{Example}

\begin{Example}
 Let $E_5=5 K_1$ be an empty graph with 5 vertices and let
  	
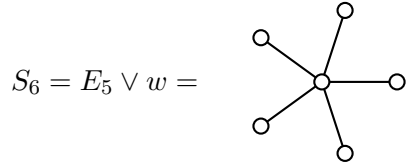
\begin{center}
			\begin{tikzpicture}[style=thick]
			   \foreach \x in {0,360/5,720/5,3*360/5,4*360/5} {
					\draw (0,0) -- (\x:1);
					\draw[fill=white] (\x:1) circle (1mm);
		};		
				\draw[fill=white] (0,0) circle (1mm);
	   			\draw (-1.5,0) node[anchor=east]{$S_6=E_5 \vee w=$};	
			\end{tikzpicture}
	\end{center}
 be a star graph with six vertices% with $\cc{S_6}=5$
 .  By Lemma \ref{lemma:vee} and Example \ref{CPrk6} it follows that 
  $$\CPrk{S_6}=\CPrk{E_5}=6>5=\cc{S_6}.$$ 
%  and take
%  $$C=\left[\begin{matrix}
%   0&0&0&0&0&0\\
%   0&0&1&1&3&3\\
%   0&1&0&3&1&1\\
%   0&1&3&0&1&1\\
%   0&3&1&1&0&3\\
%   0&3&1&1&3&0 
% \end{matrix}\right] \in \CP{6}.$$
%with $G(C)=S_6$.
%If $\CPrk{S_6}=5$, then by  Lemma \ref{cc} (1), $C = \bigoplus_{i=1}^{5} b_i \odot b_i^T \in \CP{n}$, where $(b_i)_1=(b_i)_{i+1}=0$ for $i=1,2,\ldots,5$. Now we use the same computation
%as in Example \ref{CPrk6} to gain a contradiction and thus prove that $\CPrk{S_6}>5=\cc{S_6}$. \po{Alternativa: doka?eva izrek, da se CPrk ohranja, ?e doda? v graf to?ko, ki je povezana z vsemi prej?njimi to?kami?}
%\da{DA.}
\end{Example}

\begin{Example}
 Let 
   	\begin{center}
			\begin{tikzpicture}[style=thick]
				\draw (0,0) -- (-1,-0.5) -- (-1,0.5)  -- (0,0) -- (1,0.5) -- (1,-0.5) -- (0,0);
				\draw[fill=white] (0,0) circle (1mm)  (-1,0.5) circle (1mm) (-1,-0.5) circle (1mm) (1,0.5) circle (1mm)
				    (1,-0.5) circle (1mm);
	   			\draw (-1.5,0) node[anchor=east]{$H=$};	
			\end{tikzpicture}
	\end{center}
and assume that $\CPrk{H}=\cc{H}=2$. Let 
$$D=\left[
 \begin{matrix}
  0 & 0 & 0 & 0 & 0\\
  0 & 0 & 0 & 1 & 2\\
  0 & 0 & 0 & 2 & 2\\
  0 & 1 & 2 & 0 & 0\\
  0 & 2 & 2 & 0 & 0\\
 \end{matrix}
 \right] \in \CP{5},$$ 
 and observe that $G(D)=H$. By  Lemma \ref{cca}(a),
 $$D=\left[
 \begin{matrix}
  0 & 0 & 0 & 0 & 0\\
  0 & 0 & 0 & 1 & 2\\
  0 & 0 & 0 & 2 & 2\\
  0 & 1 & 2 & 0 & 0\\
  0 & 2 & 2 & 0 & 0\\
 \end{matrix}
 \right]=\left[
 \begin{matrix}
  0 \\
  0\\
  0\\
  x\\
  y  
 \end{matrix}
 \right] \odot \left[
 \begin{matrix}
  0 &
  0&
  0 & 
  x&y
 \end{matrix}
 \right]\oplus \left[
 \begin{matrix}
  0 \\
  w\\
  t\\
  0  \\
  0
 \end{matrix}
 \right] \odot \left[
 \begin{matrix}
  0&w &
  t&
  0  & 0
 \end{matrix}
 \right].$$
Since $1=D_{2,4}=\min\{x,w\}$, it follows that $x=1$ or $w=1$. If $x=1$, then $2=D_{3,4}=\min\{1,t\}\leq 1$ and if $w=1$, then $2=D_{2,5}=\min\{y,1\}\leq 1$, both contradictions. Hence $\CPrk{H}>2=\cc{H}$.
\end{Example}

\bigskip

{\bf Acknowledgement.} The authors are grateful to the referees for their helpful remarks and suggestions that improved the presentation of this paper.

\bigskip

\bibliographystyle{plain}
\bibliography{CPrank}

\end{document}